\documentclass[egregdoesnotlikesansseriftitles,a4paper, 11pt, leqno,bibliography=totoc]{scrartcl}

\usepackage[english]{babel}			
\usepackage[utf8]{inputenc} 		
\usepackage[T1]{fontenc}		
\usepackage{fancyhdr}						
\usepackage{lastpage}
\usepackage{hyperref}						
\usepackage[dvipdf, pdftex]{graphicx}		
\usepackage{wrapfig}						
\usepackage{picinpar}						
\usepackage[margin=2ex,font=small,labelfont=bf]{caption}		
\usepackage{amsmath}
\usepackage{amsthm}
\usepackage{amssymb}
\usepackage{yfonts}							
\usepackage{setspace}						
\usepackage{tikz-cd}						
\usetikzlibrary{matrix,arrows,decorations.pathmorphing}
\usepackage{stmaryrd}						
\usepackage{dsfont}
\usepackage{graphicx}
\usepackage{stackengine}
\usepackage{mathtools}
\usepackage{stix}
\usepackage{faktor}
\usepackage{cite}
\usepackage{authblk}
\usepackage{abstract}
\usepackage[all,defaultlines=3]{nowidow}
\mathtoolsset{centercolon,showonlyrefs,showmanualtags}

\stackMath

\hypersetup{
 pdftoolbar=true, pdfmenubar=true, pdftitle={A well-posedness result for a system of cross-diffusion equations}, pdfauthor={Dominik Winkler}, linkcolor=black, citecolor=black, urlcolor=red, colorlinks=true
}

\theoremstyle{plain}
\newtheorem{theorem}{Theorem}[section]
\newtheorem*{theorem*}{Theorem}
\newtheorem{lemma}[theorem]{Lemma}
\newtheorem*{lemma*}{Lemma}

\newtheorem*{korollar*}{Korollar}
\newtheorem{proposition}[theorem]{Proposition}
\newtheorem*{proposition*}{Proposition}

\newtheorem*{satz*}{Satz}

\theoremstyle{definition}

\newtheorem*{definition*}{Definition}

\newtheorem*{beispiel*}{Beispiel}

\newtheorem*{bemerkung*}{Bemerkung}
\newtheorem*{erinnerung*}{Erinnerung}
\newtheorem{remark}{Remark}
\newtheorem*{remark*}{Remark}

\newtheorem*{assumption*}{Assumption}

\addto\captions{\renewcommand{\proofname}{\textsc{Proof}}}
\renewenvironment{proof}{{\noindent\textsc\proofname}}{\qed}

\DeclareMathOperator\supp{supp}

\newcommand{\T}{\mathbb{T}}
\newcommand{\R}{\mathbb{R}} 	
\newcommand{\Z}{\mathbb{Z}} 	
\newcommand{\N}{\mathbb{N}} 	

\newcommand{\norm}[1]{\left\lVert#1\right\rVert}
\newcommand{\supzr}{\sup \limits_{z \in \R^{n}, R^2<\infty}}
\newcommand{\intrn}{\int \limits_{\mathbb{R}^n}}

\newcommand{\inttn}{\int \limits_{\mathbb{T}^n}}
\newcommand{\grad}{\nabla}
\renewcommand{\div}{\grad\cdot}

\begin{document}

	\title{ A well-posedness result for a system of cross-diffusion equations}

	\author{
		Christian Seis \qquad Dominik Winkler}
	\affil{\textit{Institut f\"ur Analysis und Numerik, Westf\"alische Wilhelms-Universit\"at M\"unster}}
	\date{\today}
	
	\maketitle

\begin{abstract}
	This work’s major intention is the investigation of the well-posedness of  certain cross-diffusion equations 
	in the class of  bounded functions. More precisely, we show existence, uniqueness and stability of bounded weak solutions under the assumption that the system has a dominant linear diffusion. As an application, we provide a new well-posedness theory for a cross-diffusion system that originates from a hopping model with size exclusions.
	Our approach is based on a fixed point argument in a function space that is induced by suitable Carleson-type measures.
\end{abstract}


\section{Introduction}
Systems of partial differential equations with cross-diffusion have developed into a large field of research in the last decades. Cross diffusion, the phenomenon in which the gradient in the concentration of a species causes a flux of another species, appears in various applications as the modelling of population dynamics, e.g.,  \cite{burger2015lane,JungelZamponi16,SHIGESADA197983,articlemultidimensional,CHEN200639} or electrochemistry, e.g., \cite{articleelectrochemical}. Another important biological field that is mathematically described by systems with cross-diffusion are cell-sorting or chemotaxis-like problems, e.g., \cite{articlePainter,articlePainterHillen}. Chemotaxis denotes the process of cell movement provoked by chemical signals. Classical  examples involve pattern formation of bacteria, e.g., \cite{WOODWARD19952181,KELLER1971235}, or biomedical processes as tumour invasion, e.g.,\cite{GERISCH200649,articletumorgrowth}. For more detailed background information regarding the biological and modelling processes we refer the reader to \cite{Murray2002}.

In the present work we study cross-diffusion systems that are dominated by linear diffusion. More precisely, we study a system of diffusion equations that are coupled through nonlinear reaction terms
\begin{align}\label{allgemein}
\partial_t w_i - \Delta w_i = \nabla \cdot F_i(w,\nabla w) \ \ \ \ \ &\text{ in } (0,\infty)\times \T^n,& i=1,\dots,d.
\end{align}
Here, $w_i$ is the mass density, concentration or volumic fraction of the $i$th species---depending on the particular model under consideration. We choose the reaction term in divergence form for mathematical convenience. This way, the evolution is conservative, i.e., the $u_i$'s are preserved over time. If the reaction originates from (nonlinear) drift or diffusion processes in the absence of external forces it can be modelled by  $F_i(w, \nabla w) = \sum_j A_{ij}(w)\nabla w_j$ for some matrices $A(w) = \{A_{ij}(w)\}$. We suppose that the matrix is nonlinear and Lipschitz, in the sense that  
\begin{align}\label{matrixbedingung}
	|A_{ij}(w)| \lesssim |w|^{\mu} \ \ \ \text{ and } \ \ \ |A_{ij}(w)-A_{ij}(v)|\lesssim   \max\{ |w|^{\nu},|v|^{\nu}\} |w-v|, \ \ \ \forall 1\leq i,j \leq d,\end{align}
for some positive real numbers $\mu_i$ and $\nu_i$. For mathematical convenience, we choose to work on the $n$-dimensional torus $\T^n \coloneqq \faktor{\R^n}{\Z^n}$ and neglect thus any boundary effects. Further we equip system \eqref{allgemein} with initial data $h_1,\dots, h_d$.

In its full generality, it is very challenging to study the well-posedness for \eqref{allgemein} without further assumptions. In the present work, our goal is to exploit the particular structure of the nonlinearity, in order to derive a well-posedness result for weak solutions with small initial data.
\begin{theorem}\label{Tallg}
  	For every sufficient small set of initial data $h=(h_1,\dots,h_d)$, there exists a solution $w=(w_1,\dots,w_d)$ to system \eqref{allgemein}. The solution is unique in the class of functions satisfying
  	\begin{align}\label{3}
  	\norm{w}_{L^\infty}+\sup_t \sqrt{t} \norm{\nabla w(t)}_{L^{\infty}} \lesssim \norm{h}_{L^{\infty}}.
  	\end{align}
  	Moreover, if $\tilde w=(\tilde{w}_1,\dots,\tilde{w}_d)$ is another set of solutions with initial data $\tilde h=(\tilde{h}_1,\dots,\tilde{h}_d)$, it holds that
  	\begin{align}\label{11}
  		  \norm{w-\tilde{w}}_{L^{\infty}} \lesssim   \norm{h-\tilde{h}}_{L^{\infty}}.
  	\end{align}
\end{theorem} 
In fact, our result is a bit stronger, in the sense that we consider a class of functions that is slightly larger than \eqref{3}. The corresponding function space is defined via suitable Carleson measures or, a little more accurate,  via $L^{\infty}$ norms of certain Hardy--Littlewood maximal functions. We will discuss these spaces and their orgin later in Section \ref{section1.2}.
A more detailed version of Theorem \ref{Tallg} will be given in Theorem \ref{thm1}. Moreover, we will see that our solutions $w$ are of class $C^m$, class $C^{\infty}$ or analytic if $F(w,\nabla w)$ is of the according class $C^m$, $C^{\infty}$ or $C^{\omega}$ as well. Estimates analogous to the gradient estimate \eqref{3} hold true also for any derivatives in time and space,
\[
\sup_{t,x}\,  {t }^{k + \frac{|\beta|}2} |\partial_t^k\partial_x^{\beta} w(t,x)| \lesssim \norm{h}_{L^{\infty}},
\]
for any $k\in \N_0$ and any $\beta\in \N_0^n$ such that the derivatives exist,  see Theorem \ref{thm2}. 
\begin{remark*}
	In this work we write $x \lesssim y$, if the inequality only holds true up to a positive constant $C< \infty$. For the arguments used, the precise values of these constants are irrelevant.
\end{remark*}

 The reason, why we choose to work in the setting of bounded functions is particularly motivated by the following specific example, which apparently belongs to the class of cross-diffusion systems modelled in \eqref{allgemein}, \eqref{matrixbedingung}. We study a  cross-diffusion system that can be modelled by a multi-dimensional advection-diffusion equation with linear drift and diffusion matrices, 
  \begin{align}\label{ursprungssystem}
 	\partial_t u_i = \nabla \cdot \Big[\sum \limits_{j=1, j \neq i}^d K_{ij}(u_j\nabla u_i - u_i \nabla u_j)\Big] \ \ \ \ \ &\text{ in } (0,\infty)\times \T^n,& i=1,\dots,d.
 \end{align}
This system describes the evolution of $d$ different species, and $u_i(t,x)$ plays the role of the density or volumic fraction of the $i$th species at time $t$ and point $x$. The $K_{ij}$'s are the cross-diffusion coefficients, which relate the gradient of the $j$th species' concentration with the flux of the $i$th species' concentration. 
To illustrate the structure of \eqref{ursprungssystem}, we note that the evolution of the $i$th species can be rewritten as the linear conservative advection-diffusion equation
\[
\partial_t u_i + \div(b u_i) = \div(a\grad u_i),
\]
in which the diffusion coefficient $a$ is proportional to the concentration of the concurrent species, while the advecting velocity field $b$ is linearly dependent on  their concentration gradients. The system can be derived  as a formal limit from a hopping model with size exclusion, see \cite{articlehopmodel}.  It   was recently studied mathematically in    \cite{articlepietschmann}.

   Since the solution $u_i$ for $i=1,\dots,d$ represents the volumic fraction of the $i$th species, it is reasonable to  demand the solutions to partition unity,
 \begin{align}\label{bedingungsumme}
 \forall 1\leq i \leq d,\ \ \ 	u_i(t,x) \geq 0 \ \ \ \text{ and }   \ \ \ 	\sum \limits_{i=1}^d u_i(t,x)=1 \ \ \ \text{ in } (0,\infty) \times \T^n.
 \end{align} 
 The same condition has thus to be satisfied by the initial data $g= (g_1,\dots,g_d)$, that is,
  \begin{align}\label{bedingungang}
 	\forall 1\leq i \leq d,\ \ \ 	g_i(x) \geq 0 \ \ \ \text{ and }   \ \ \ 	\sum \limits_{i=1}^d g_i(x)=1 \ \ \ \text{ in }\T^n.
 \end{align}
In order to ensure that the partiton condition in \eqref{bedingungsumme} is satisfied, even on a formal level, it is necessary to impose that the diffusion coefficients are symmetric in the sense that
\begin{equation}\label{2}
K_{ij} = K_{ji}\quad\mbox{for all }1\leq i \neq j\leq d.
\end{equation}

Even for this specific model, proving uniqueness and pointwise bounds as in \eqref{bedingungsumme} is rather challenging. In the following, inspired by \cite{articlepietschmann}, we will restrict our attention to the case, in which the cross-diffusion coefficients $K_{ij}$ satisfy certain closeness assumptions. This way, despite the constraint in \eqref{bedingungsumme}, we are in a situation in which our system under consideration is equivalent to  that in \eqref{allgemein}, \eqref{matrixbedingung}, and thus Theorem \ref{Tallg} applies.

To be more specific, thanks to the partition condition in \eqref{bedingungsumme}, we can elegantly generate a linear diffusion term in \eqref{ursprungssystem},
 \begin{align}\label{pietschmannsystem}
 	\partial_t u_i - K \Delta u_i = \nabla \cdot \Big[ \sum \limits_{j=1, j \neq i}^d( K_{ij}-K)(u_j \nabla u_i - u_i \nabla u_j) \Big] &\text{ in } (0,\infty)\times \T^n,& i=1,\dots,d,
 \end{align}
for any positive constant $K$. In order to treat the right-hand side as a perturbation, we have to assume that the coefficients are sufficiently close  to each other. This is achieved, for instance, by choosing 
\[
K \coloneqq \frac{1}{2} \big(  \max \limits_{1\leq i \neq j \leq d} K_{ij} +  \min \limits_{1\leq i \neq j \leq d} K_{ij}\big),
\]
and demanding that 
\begin{equation}\label{1}
\max \limits_{1\leq i \neq j \leq d} |K_{ij}-K|\ll K/d.
\end{equation}
This assumption enables us to translate \eqref{ursprungssystem} or \eqref{pietschmannsystem} into a diffusion-dominant system, see Section \ref{section1.2}.

Theorem \ref{Tallg} provides us, due to scaling argument, with  a unique solution to \eqref{pietschmannsystem} in the class of functions satisfying
\begin{equation}\label{10}
	\|u\|_{L^{\infty}}+ \sup_t \sqrt{Kt} \norm{\nabla u(t)}_{L^\infty} \lesssim 1.
\end{equation}
In fact, we will see that this system can be transferred back into the original cross-diffusion system \eqref{ursprungssystem}, \eqref{bedingungsumme}. We thus have the following well-posedness result.

\begin{theorem}
\label{T1}
Suppose that the coefficients $K_{ij}$ are symmetric and sufficiently close  to each other in the sense of \eqref{1} and \eqref{2}. Then, for every set of initial data $g_1,\dots,g_d $ satisfying \eqref{bedingungang}, there exists a  smooth solution  $u_1,\dots,u_d$ to the cross-diffusion system \eqref{ursprungssystem}, \eqref{bedingungsumme}. This solution is unique in the class of functions satisfying \eqref{10}. Moreover, solutions are stable in the sense of \eqref{11}.
\end{theorem}

\begin{remark}
We remark that Theorem \ref{T1} (as Theorem \ref{thm1} and Theorem \ref{thm3} below) is valid also for more general classes of cross-diffusion coefficients that vary in space and time, $K_{ij}=K_{ij}(t,x)$, as long as \eqref{1} and \eqref{2} remain true. 
\end{remark}

We note that solutions are automatically bounded thanks to the modelling assumption \eqref{bedingungsumme}, which makes $L^{\infty}$ a natural space for the study of well-posedness. Moreover, the gradient estimate \eqref{3} or \eqref{10} is natural in this perturbative setting \eqref{1}, as it is the standard gradient estimate for the homogeneous heat equation with $L^{\infty}$ data---observe that the control over the gradient deterioates as $t\to 0$ with a rate proportional to the diffusion length. In this sense, we consider the conditions for well-posedness imposed in the present paper as optimal. Our  well-posedness result  for the system under consideration improves upon earlier results which require the solutions and data to be of higher regularity \cite{articlepietschmann}.

 We finally remark that, in general, the analytic treatment of many cross-diffusion problems in the form of
\begin{align}\label{juengelsystem}
\partial_t u - \nabla \cdot \big( a(u)\nabla u\big) = f(u),
\end{align}
can be very challenging, since the diffusion matrix $a(u)$ neither has to be symmetric nor positive definite, which makes it hard to ensure such modelling assumptions as in  (\ref{bedingungsumme}). Another difficulty lies in the absence of a maximum principle or general parabolic regularity theory, if the diffusion matrix is not diagonal. Sufficient conditions for the global existence of weak or strong solutions of nonlinear parabolic equations are obtained, for instance, in \cite{LadyzenskayaSolonnikovUralceva68,Amann89,Pierre10,juengel,ChenJungel19}. The problem of uniqueness is in general much harder. For mildly coupled cross-diffusion equations uniqueness has been proved by  duality methods \cite{Jungel00,DiazGalianoJungel01,MiuraSugiyama14}. In some situations, the structure of the equations also allows for the application of entropy methods \cite{JungelZamponi16,ZamponiJungel17,ChenJungel18}. 
We finally mention results on weak-strong uniqueness in \cite{articlepietschmann,Fischer17,ChenJungel19}.

The paper is organized as follows: In Section \ref{section1.2}, we introduce and discuss the precise function spaces in which we establish well-posedness. Section \ref{linearproblem} is devoted to the study of the linear problem in these spaces. In Section \ref{chapter3} we come back to the nonlinear problem and provide the proofs of the main theorems.


\section{Reformulation and results}\label{section1.2}

The systems that we investigate in this work can be considered as nonlinear perturbations of multi-dimensional heat equations. Moreover,  the particular (semilinear) structure of the nonlinearity considered in \eqref{allgemein}, more precisely, the properties formulated in \eqref{matrixbedingung}, which are in turn motivated my the particular example mentioned in \eqref{ursprungssystem} or \eqref{pietschmannsystem}, lead to the study of bounded solutions to the respective equations in a natural way. Indeed, for any well-behaved norm $\|\cdot\|$ for which we have maximal regularity estimates for the heat equation, we expect that 
\[
\|\grad w\|  \lesssim \|F(w,\grad w)\| + \|h\| \lesssim \||w|^{\mu}\grad w\| + \|h\|\lesssim \|w\|_{L^{\infty}}^{\mu} \|\grad w\|+ \|h\|
\]
by the virtue of \eqref{matrixbedingung}, and the nonlinear term on the right-hand side can be absorbed into the left-hand side provided that $\|w\|_{L^{\infty}}$ is sufficiently small. We are thus led to considering $\|h\| = \|h\|_{L^{\infty}}$ in the case of the initial datum---a choice that is consistent with the partition of unity condition imposed in \eqref{bedingungsumme}, \eqref{bedingungang}. The space-time maximal regularity norm has to be accordingly scale-invariant. Motivated by \cite{KOCH200122}, we use the following (semi-)norms, that are motivated by Carleson-measure characterizations of the BMO space, see Theorem 3 of Chapter 4.4 in \cite{10.2307/j.ctt1bpmb3s}.

Given functions $w:(0,\infty) \times \T^n \rightarrow \R$ and $F: (0,\infty) \times \T^n \rightarrow \R^n$ and $p\in(1,\infty)$, we define
\begin{align*}
& \norm{w}_{X^p}  \coloneqq \| w\|_{L^{\infty} } + \norm{w}_{\dot X^p} ,\qquad 	\norm{w}_{\dot X^p} \coloneqq \sup\limits_{z\in \T^n, R^2<\infty} R \big( \fint \limits_{Q_R(z)} |\nabla w(t,x)|^p dx dt \big)^{\frac{1}{p}} ,\\
&	\norm{F}_{Y^p} \coloneqq \sup \limits_{z \in \T^n, R^2 < \infty} R \big(\fint \limits_{Q_R(z)} |F(t,x)|^pdxdt\big)^{\frac{1}{p}},
\end{align*}
where $Q_R(z) \coloneqq [\frac{R^2}{2},R^2] \times B_R(z) \subseteq \R \times \R^n$. If necessary, we identify $w$ or $F$ with its spatial periodic extension. Based on these norms we define two Banach spaces  $X^p$ and $Y^p$ by
	\begin{align*}
&	X^p\coloneqq \bigl\{ w:(0,\infty) \times \T^n \rightarrow \R \mid  \norm{w}_{X^p} < \infty \bigr\}
	\text{ and }\\
&	Y^p \coloneqq \bigl\{ F : (0,\infty) \times \T^n \rightarrow \R^n \mid \norm{F}_{Y^p} < \infty \bigr\}.
	\end{align*}
The underlying concept of using such norms   was introduced   in \cite{KOCH200122}, in order to prove well-posedness for the Navier-Stokes equations with small initial data in BMO$^{-1}$.
This concept was further developed in order to establish existence and uniqueness results for various (degenerate) parabolic equations, including geometric flows with rough data \cite{koch2012,Wang11}, the porous medium equation \cite{Kienzler16},   the thin film equation \cite{John15,MR3785606}, and the Landau--Lifshitz--Gilbert equation \cite{GutierrezdeLaire19}.

By a slight abuse of notation, we generalize these norms and spaces to vector or matrix valued functions by setting 
\[
\norm{w}_{X^p} \coloneqq \max \limits_{i=1,\dots,d} \norm{w_i}_{X^p},\quad \norm{F}_{Y^p} \coloneqq \max \limits_{i=1,\dots,d} \norm{F_i}_{Y^p},\quad  \norm{h}_{L^{\infty}} \coloneqq \max \limits_{i=1,\dots,d} \norm{h_i}_{L^{\infty}},
\]
 for tuples $w=(w_1,\dots,w_d)$, $F=(F_1,\dots,F_d)$, and $h=(h_1,\dots,h_d)$. 

We are now in the position to present our first result (Theorem \ref{Tallg}) in a more precise manner.

\begin{theorem}\label{thm1}
Suppose that \eqref{matrixbedingung} holds and let $p > n+2$. There exist $\delta_0 > 0$ and $C >0$ such that for every $\delta \leq \delta_0$ and every initial data $h  $ with $\norm{h}_{L^{\infty} }\leq \delta$, there exists a unique solution $w$ to the system \eqref{allgemein} in the class $\norm{w}_{X^p} \leq C\delta$. Moreover, if $\tilde w$ is another solution with initial datum $\tilde h$, it holds that
\begin{equation}\label{12}
\|w-\tilde w\|_{X^p} \lesssim \|h-\tilde h\|_{L^{\infty}}.
\end{equation}
\end{theorem}
Under additional assumptions concerning the nonlinearity $F$, we are able to show higher regularity of the solutions.
\begin{theorem}\label{thm2}
	Let $A$ be of class $C^m$, of class $C^\infty$ or analytic. Then there exist and $\delta_0>0$, maybe even smaller than needed in Theorem \ref{thm1}, such that the dependence of the solution $w$ from Theorem \ref{thm1} on the initial data $h$ is of class $C^m$, class $C^\infty$ or analytic. Further the solution is of class $C^m$, class $C^\infty$ or analytic in time and space. For every $k \in \N_0$ and every multiindex $\beta \in \N_0^n$ such that the derivative exists, it holds that	\begin{align}\label{7}
	\sup_{t,x} t^{k+\frac{|\beta|}2}\big|\partial_t^k\partial_x^{\beta}w(t,x)\big| \lesssim \norm{h}_{L^{\infty} }.
	\end{align}
	In the analytic case there exist constants $\Gamma >0$ and $C>0$ independent of $k$ and $\beta$ such that
	\begin{align}\label{8}
	\sup_{t,x}t^{k+\frac{|\beta|}2}\big|\partial_t^k\partial_x^{\beta}w(t,x)\big| \leq C \Gamma^{l+|\beta|}k!\beta!\norm{h}_{L^{\infty} }.
	\end{align}
	for every $k \in \N_0$ and every multiindex $\beta \in \N_0^n$.
\end{theorem}

We finally turn to the explicit system given in \eqref{ursprungssystem}, \eqref{bedingungsumme} and show how it fits into the general framework considered in Theorems \ref{thm1} and \ref{thm2}. We have already seen that under the partition condition in \eqref{bedingungsumme}, \eqref{ursprungssystem} is equivalent to \eqref{pietschmannsystem}.
Our goal is to transfer the latter into a diffusion-dominated system with small initial data. By rescaling time, the diffusivity constant on the left-hand side can be absorbed into the cross-diffusion coefficients, that is, we consider 
\begin{align}\label{pertubation}
\partial_t u_i - \Delta u_i = \nabla \cdot \Big[ \sum \limits_{j=1, j \neq i}^d \delta_{ij}(u_j \nabla u_i - u_i \nabla u_j) \Big]&\text{ in } (0,\infty)\times \T^n,& i=1,\dots,d,
\end{align}
with coefficients $\delta_{ij} \coloneqq \frac{K_{ij}}{K}-1$. At this point we note that the scaling factor $K$ has to be positive. 
The closeness condition \eqref{1} now translates into the smallness condition $\delta \coloneqq \max \limits_{1\leq i \neq j \leq d} |\delta_{ij}| \ll \frac{1}{d}$ on the new coefficients. We now use the nonlinearity of the equation to shift the smallness condition further to the initial datum. This is achieved by setting $w_i\coloneqq\delta u_i$ and $h_i\coloneqq\delta g_i$. The new  partition conditions are thus
\begin{align}\label{bedingungsummeW}
\forall 1\leq i \leq d,\ \ \ 	w_i(t,x) \geq 0 \ \ \ \text{ and }   \ \ \ 	\sum \limits_{i=1}^d w_i(t,x)=\delta \ \ \ \text{ in } (0,\infty) \times \T^n,
\end{align}
and
\begin{align}\label{bedingunganh}
\forall 1\leq i \leq d,\ \ \ 	h_i(x) \geq 0 \ \ \ \text{ and }   \ \ \ 	\sum \limits_{i=1}^d h_i(x)=\delta \ \ \ \text{ in }\T^n,
\end{align}
and the cross-diffusion equations become
\begin{align}\label{bewiesenessystem}
\begin{cases}
\partial_t w_i - \Delta w_i = \nabla \cdot \Big[ \sum \limits_{j=1, j \neq i}^d \alpha_{ij} (w_j\nabla w_i - w_i \nabla w_j)\Big] &  \text{ in } \left(0,\infty\right) \times \T^n,\\
w_i(0,\cdot)=h_i & \text{ in } \T^n,
\end{cases} \qquad & i = 1,\dots d,
\end{align}
where the $\alpha_{ij}$'s are given by $\frac{\delta_{ij}}{\delta}$, and are thus bounded, $|\alpha_{ij}|\leq 1$.

\begin{remark*}
	To be accurate, we have to exclude the case $K_{ij} \equiv K$. Since we would obtain $\delta =0$, the change of variables would not be permitted. However, this is not a significant restriction, because the the cross-diffusion system would untangle into a system of $d$ independent heat equations, which is much easier to solve.
\end{remark*} 

We see that \eqref{bewiesenessystem} has the same structure as our general model \eqref{allgemein}, where $A(w)$ is given by
\begin{align*}
	A_{ij}(w)=\begin{cases} \sum \limits_{k=1, j \neq i}^d \alpha_{ik}w_k &\mbox{if }i=j,\\
	  -\alpha_{ij}w_i &\mbox{if }i\not=j.
	 \end{cases}
\end{align*}
Apparently, \eqref{bedingunganh} provides an upper bound for the initial data and the nonlinearity $A(w)$ depends analytically on $w$. We are thus allowed to apply Theorems \ref{thm1} and  \ref{thm2} and obtain well-posedness for \eqref{bewiesenessystem} in the class $X^p$ together with analyticity in time and space and analytic dependence on the initial data. It only remains to verify that solutions obey the partition condition \eqref{bedingungsummeW}, the argument of which we provide in Section \ref{chapter3}, following \cite{articlepietschmann}. Our result for the cross-diffusion system \eqref{ursprungssystem}, \eqref{bedingungsumme}, or equivalently, \eqref{bewiesenessystem}, \eqref{bedingungsummeW},  is thus the following.

\begin{theorem}\label{thm3}
Suppose that the coefficients $K_{ij}$ are symmetric   in the sense of \eqref{1}. Let  $p > n+2$ be given. There exist $\delta_0 > 0$ and $C >0$ such that for every $\delta \leq \delta_0$ and every $h \in L^{\infty}$ with \eqref{bedingunganh}, there exists a unique solution $w$ to the system (\ref{bewiesenessystem}),  (\ref{bedingungsummeW}) in the class $\norm{w}_{X^p} \leq C\delta$. Moreover, the solution depends analytically on time, space and the initial data, and estimates \eqref{8} and \eqref{12} hold.
\end{theorem}

\section{Linear theory}\label{linearproblem}
Our proof of  Theorem \ref{thm1} is based on a fixed point argument. We will thus start with the study of the linear problem.
Our goal in this section is the following maximal regularity estimate.
\begin{proposition}\label{proposition}
	Let $w$ be a solution of the inhomogeneous heat equation
	\begin{align} \label{heateq}
		\begin{cases}
		\partial_t w - \Delta w = \nabla \cdot F & \text{ in } \left(0, \infty\right)\times \T^n\\
		w(0,\cdot)=h & \text{ in } \T^n.
		\end{cases}
	\end{align}
	Then it holds that 
	\[
	\norm{w}_{X^p} \lesssim \norm{F}_{Y^p} + \norm{h}_{L^{\infty}}.
	\]
\end{proposition}

The argumentation for establishing this proposition is similar to those in \cite{koch2012,Kienzler16,John15,MR3785606}.

It will be convenient to translate the problem onto the full space by extending all involved functions periodically from $\T^n$ to $\R^n$. It is clear that the corresponding norms remain unchanged under periodic extension. 

We denote the heat kernel in $\R^n$ by $\Phi$, i.e., $\Phi(t,x)=(4\pi t)^{-\frac{n}{2}} e^{-\frac{|x|^2}{4t}}$, so that solutions to \eqref{heateq} have the representation
\begin{align*}
w(t,x) &= \int \limits_{\R^n} \Phi(t,x-y)h(y)dy + \int \limits_0^t \int \limits_{\R^n} \grad\Phi(t-s,x-y)\cdot F(s,y)dyds=: \tilde w(t,x) + \hat w(t,x).
\end{align*}
We will estimate the homogeneous part $\tilde w$ and the inhomogeneous part $\hat w$ separately.
Before doing so, we recall a standard estimate on the gradient of the heat kernel.
\begin{lemma} \label{lemma2.1}
For every $p\in[1,\infty]$, it holds that
	\[
		\norm{\nabla \Phi (t, \cdot)}_{L^p(\mathbb{R}^n)} ~ \lesssim ~ t^{-\frac{n}{2}-\frac{1}{2}+\frac{n}{2p}}.
	\]
\end{lemma}
We provide the simple proof for the convenience of the reader.

\begin{proof}\textsc{.}
	For any $\alpha >0$, the function $y^{\alpha}e^{-\frac{y}{2}}$ is bounded on $[0,\infty)$ and thus 
	\[
	y^{\alpha}e^{-y} = y^{\alpha}e^{-\frac{y}{2}}e^{-\frac{y}{2}} \lesssim e^{-\frac{y}{2}}.
	\]
	Using $y=\frac{|x|^2}{4t}$ and $\alpha = \frac{1}{2}$, we thus obtain the pointwise estimate
	\begin{align*}
		|\nabla \Phi(t,x)| = \frac{1}{(4 \pi t)^{\frac{n}{2}}} \frac{|x|}{2t} e^{-\frac{|x|^2}{4t}} \lesssim \frac{1}{t^{\frac{n+1}{2}}}e^{-c\frac{|x|^2}{t}}.
	\end{align*}
	This proves the case $p=\infty$. For smaller values of $p$, using a chance of variables, we compute
	\begin{align*}
		\norm{\nabla \Phi (t, \cdot)}_{L^p(\mathbb{R}^n)}^{p}  \lesssim t^{-\frac{np}{2}-p} \int \limits_{\R^n}|x|^p e^{-\frac{p|x|^2}{4t}}dx \lesssim t^{-\frac{np}{2}-p+\frac{p}{2}+\frac{n}{2}}\int \limits_{\R^n}|y|^p e^{-\frac{|y|^2}{4}}dy \lesssim t^{-\frac{np}{2}-\frac{p}{2}+\frac{n}{2}}.
	\end{align*}
	This proves the lemma.
\end{proof}

We first turn to the estimate of the solution to the homogeneous problem $\tilde w$. 

\begin{lemma} \label{heateqhomogen}
It holds that $ \norm{\tilde{w}}_{X^p} \lesssim \norm{{h}}_{L^{\infty}}$.
\end{lemma}

\begin{proof}\textsc{.}
The maximum principle for the heat equation immediately implies the bound on the $L^{\infty}$ norm of $\tilde w$.
	In order to estimate the Carleson measure part of the $X^p$ norm, we observe that
	\begin{align*}
	\sup \limits_{x \in \R^n, t\in [\frac{R^2}{2},R^2]} | \nabla \tilde{w}(t,x)| &\leq \sup \limits_{x \in \R^n, t\in [\frac{R^2}{2},R^2]} \int \limits_{\R^n} | \nabla \Phi(t,y) {h}(x-y)|dy \\&\leq \sup \limits_{ t\in [\frac{R^2}{2},R^2]} \norm{\nabla \Phi(t, \cdot)}_{L^1(\R^n)} \norm{{h}}_{L^{\infty}(\R^n)}\\
	& \lesssim \frac{1}{\sqrt{R^2}} \norm{{h}}_{L^{\infty}(\R^n)},
	\end{align*}
	due to Lemma \ref{lemma2.1}. Using this estimate we get
	\begin{align*}
	\sup \limits_{z \in \R^n, R^2<\infty} R \Big( \fint \limits_{Q_R(z)} |\nabla \tilde{w}|^p dx dt  \Big)^{\frac{1}{p}} \leq \supzr R \norm{\nabla \tilde{w}}_{L^{\infty}(Q_R(z))} \lesssim R \frac{1}{R} \norm{{h}}_{L^{\infty}(\R^n)},
	\end{align*}
	which proves the Lemma.
\end{proof}

\begin{lemma} \label{lemmainhom}
It holds that  $\norm{\hat {w}}_{X^p} \lesssim \norm{{F}}_{Y^p}$.
\end{lemma}

\begin{proof}\textsc{.}
	We start with the bound on the $L^{\infty}$-norm of $\hat{w}$. We set $R=\sqrt t$ and split	the space-time integral into a diagonal and an off-diagonal part,
	\begin{align*}
	|\hat{w}(t,x)| &\leq \big| \int \limits_{Q_R(x)} \nabla \Phi(t-s,x-y)\cdot {F}(s,y)dyds \big|\\
	&\qquad + \big| \int  \limits_{[0,R^2]\times \R^n \setminus Q_R(x)} \nabla \Phi(t-s,x-y) \cdot {F}(s,y)dyds\big|\\
	& \eqqcolon A+B.
	\end{align*}
	
To bound the diagonal part, we use	 Hölder's inequality and get 
\[
A \leq \norm{\nabla \Phi}_{L^q([0, \frac{R^2}{2}]\times B_R(0))} \norm{{F}}_{L^p(Q_R(x))},
\]
for any H\"older conjugates $p$ and $q$.	We have to choose  $q$ small enough such that the $L^q$-norm of $\nabla \Phi$ is finite. From Lemma \ref{lemma2.1} we get
	\begin{align*}
	\Big( \int \limits_0^{\frac{R^2}{2}}  \int \limits_{B_R(0)} | \nabla \Phi|^qdxds\Big)^{\frac{1}{q}} \lesssim \big( \int \limits_0^{\frac{R^2}{2}} t^{-\frac{nq}{2}- \frac{q}{2}+ \frac{n}{2}}dt\big)^{\frac{1}{q}}.
	\end{align*}
	The right-hand side is finite if and only if $-\frac{nq}{2}- \frac{q}{2}+ \frac{n}{2} >-1$, which is equivalent to  requiring that $p >n+2$, as in the assumption of  Theorem \ref{thm1}. We evaluate the integral on the right-hand side and obtain for the diagonal part of $\hat w$ that
	\begin{align*}
	A \lesssim R^{1-\frac{n+2}{p}}\norm{ {F}}_{L^p(Q_R(x))} \lesssim \supzr R \big( \fint \limits_{Q_R(z)} | {F}|^p\big)^{\frac{1}{p}} \leq \norm{ {F}}_{Y^p}.
	\end{align*}
	
	Let us now consider the off-diagonal term $B$. Applying elementary arguments, we observe
	\begin{equation}\label{100}
	B \lesssim \int \limits_{[0,R^2]\times \R^n} \frac{1}{R^{n+1}}e^{-c\frac{|x-y|}{R}}|{F}(s,y)|dyds.
	\end{equation}
		In order to control the term on the right by the Carleson measure expression which defines the $Y^p$ norm, we have to invoke a covering argument. Using the triangle inequality and the fact that $\sum \limits_{\tilde{x}\in R \cdot\Z^n} e^{-c\frac{|x-\tilde{x}|}{R}}$ is controlled by a constant only depending on the dimension $n$, we notice that
	\begin{align*}
	B& \leq \sum \limits_{m=0}^{\infty} \sum \limits_{\tilde{x} \in R\cdot \Z^n} \int \limits_{R^2\cdot 2^{-(m+1)}}^{R^2 \cdot 2^{-m}} \int \limits_{B_R(\tilde{x})} e^{-c\frac{|x-y|}{R}} \frac{1}{R^{n+1}}| {F}(s,y)| dyds \\
	& \lesssim \sum \limits_{m=0}^{\infty} \sum \limits_{\tilde{x} \in R \cdot \Z^n} e^{-\frac{|x-\tilde{x}|}{R}}\int \limits_{R^2\cdot 2^{-(m+1)}}^{R^2 \cdot 2^{-m}} \int \limits_{B_R(\tilde{x})} \frac{1}{R^{n+1}}| {F}(s,x-z)|dzds\\& \lesssim
	\sum \limits_{m=0}^{\infty} \sup \limits_{\tilde{x} \in R \cdot \Z^n} \int \limits_{R^2\cdot 2^{-(m+1)}}^{R^2 \cdot 2^{-m}} \int \limits_{B_R(\tilde{x})} \frac{1}{R^{n+1}}| {F}(s,y)|dyds.
	\end{align*}
	Now we claim, that there exists a constant $ 0 < \gamma <1$ independent of $m$, such that 
	\begin{equation}\label{13}
	\int \limits_{R^2\cdot 2^{-(m+1)}}^{R^2 \cdot 2^{-m}} \int \limits_{B_R(\tilde{x})} \frac{1}{R^{n+1}}| {F}(s,y)|dyds \lesssim \gamma^m \norm{{F}}_{Y^p}.
	\end{equation}
This estimate directly implies that $B \lesssim \norm{{F}}_{Y^p}$ as a conclusion from the geometric series' convergence, which in turn establishes the control of the $L^{\infty}$ norm as desired.
	
	To prove the claim in \eqref{13}, we have to refine the spatial covering. Indeed, we cover the set $(R^2 \cdot 2^{-(m+1)}, R^2 \cdot 2^{-m})\times B_R(\tilde{x} )$ by about $2^{\frac{mn}{2}}$ many cylinders $Q_m(z)$ of the form $Q_m(z) \coloneqq (R^2 \cdot 2^{-(m+1)}, R^2 \cdot 2^{-m})\times B_{R\cdot 2^{-\frac{m}{2}}}(z)$. We now obtain 
	\begin{align*}	
	\int \limits_{R^2\cdot 2^{-(m+1)}}^{R^2\cdot 2^{-m}} \int \limits_{B_R(\tilde{x})}\frac{1}{R^{n+1}}| {F}|dyds  &\lesssim 2^{\frac{nm}{2}} \frac1{R^{n+1}}\sup \limits_z
 \norm{ {F}}_{L^1(Q_m(z))}\\
	&\lesssim  2^{\frac{nm}{2}}  \frac{1}{R^{n+1}} \big( R \cdot 2^{-\frac{m}{2}} \big)^{n+2}\frac1{R2^{-\frac{m}2}} \norm{ {F}}_{Y^1} \\
	&=2^{-\frac{m}{2}}  \norm{ {F}}_{Y^1}.	\end{align*}
	Since $\|F\|_{Y^1}\le \|F\|_{Y^p}$ by Jensen's inequality, we see that $\gamma = 2^{-\frac{1}{2}}$ is a valid constant.
	
	It remains to estimate the Carleson measure part of the $X^p$ norm. Again, we consider separately the diagonal and the off-diagonal contribution, this time, however, by distinguishing the two cases $\supp( {F})\subset [0,\infty)\times \R^n \setminus Q_R(z)$ and $\supp( {F})\subseteq \check{Q}_R(z) \coloneqq (\frac{R^2}{4},R^2)\times B_{2R}(z)$. The general case is obtained by a standard cut-off procedure via the triangle inequality.

	\textit{Case 1: We assume $\supp( {F})\subset [0,\infty)\times \R^n \setminus Q_R(z)$.}\\
	Then we get for the absolute value of $\nabla \hat{w}(x,t)$ with $R^2=t$:
	\begin{align*}
	|\nabla \hat{w}(t,x)|& \leq \intrn \int \limits_0^{R^2}|\nabla^2\Phi(t-s,x-y)| | {F}(s,y)|dyds \\
	&= \int \limits_{ [0,R^2]\times \R^n \setminus Q_R(x)} |\nabla^2\Phi(t-s,x-y)| | {F}(s,y)|dyds\\
	&\lesssim \int \limits_{ [0,R^2]\times \R^n } \frac{1}{R^{n+2}}e^{-c\frac{|x-y|}{R}}| {F}(s,y)|dyds.
	\end{align*}
Up to a factor $1/R$, the term on the right-hand side is precisely the term that we hat to bound in our previous argument for $B$, see \eqref{100}. We thus find
	\begin{align}
	|\nabla \hat{w}(t,x)| \lesssim	\frac{1}{R}\norm{ {F}}_{Y^p},
	\end{align}
and averaging over the cylinder $Q_R(z)$ gives
	\begin{align*}
			\sup \limits_{z \in \R^n, R^2<T} R \Big( \fint \limits_{Q_R(z)} |\grad \hat {w}|^p dx dt  \Big)^{\frac{1}{p}} \lesssim R \cdot \norm{\nabla \hat{w}}_{L^{\infty}} \lesssim  \norm{ {F}}_{Y^p},
	\end{align*}
	as desired.
	
	\textit{Case 2: We assume $\supp({F})\subseteq \check{Q}_R(z)  $}.\\
Our argumentation for this case is based on the maximal regularity estimate for the heat equation with forcing in divergence form,
	\begin{align}\label{abs}
		\norm{\nabla \hat {w}}_{L^p((0,\infty)\times \R^n)} \lesssim \norm{ {F}}_{L^p((0,\infty) \times \R^n)}.
	\end{align}
Restriction on the support of the forcing, we get $\norm{\nabla \hat{w}}_{L^p(Q_R(z))} \lesssim \norm{ {F}}_{ L^p(\check{Q}_R(z))}$. We can cover $\check{Q}_R(z)$ by $Q_R(z)\cup Q_{2R}(z) \cup Q_{\frac{R}{\sqrt{2}}}(z)$ and thus we obtain
	\begin{align*}
	\MoveEqLeft[5]
		R^{1-\frac{n+2}{p}}\norm{\nabla \hat{w}}_{L^p(Q_R(z))} \leq R^{1-\frac{n+2}{p}}\big( \norm{ {F}}_{L^p(Q_R(z))} + \norm{ {F}}_{L^p(Q_{2R}(z))+ \norm{ {F}}_{L^p(Q_{R/\sqrt{2}}(z))}}\big)\\
		 &\lesssim R^{1-\frac{n+2}{p}} \norm{ {F}}_{L^p(Q_R(z))} + (2R)^{1-\frac{n+2}{p}} \norm{ {F}}_{L^p(Q_{2R}(z))} + \big(\frac{R}{\sqrt{2}}\big)^{1-\frac{n+2}{p}} \norm{ {F}}_{L^p(Q_{R/\sqrt{2}}(z))} \\
		 & \lesssim \norm{ {F}}_{Y^p}. 
	\end{align*}
	Maximizing in $R$ and $z$ on the left-hand side yields the missing estimate.
\end{proof}

\section{The nonlinear problem}\label{chapter3}
In this section we want to prove Theorems \ref{thm1}, \ref{thm2}, and \ref{thm3}. 
Our first concern is the well-posedness of the system \eqref{allgemein} under the assumption \eqref{matrixbedingung} on the nonlinearity, which we derive   by a fixed point argument. To apply this argument we need the following lemma.
\begin{lemma}\label{lemmakontraktion}
It holds that
	\begin{align}\label{abschaetzung}
		\norm{F(v,\nabla v)-F(w,\nabla w)}_{Y^p} \lesssim d \max \{\norm{v}_{X^p}^{\mu},\norm{w}_{X^p}^{\mu}, \norm{v}_{X^p}^{\nu+1}, \norm{w}_{X^p}^{\nu+1} \} \norm{v-w}_{X^p}
	\end{align}
	and
\begin{align}\label{5}
	\norm{F(v,\nabla v)}_{Y^p} \lesssim d \norm{v}_{X^p}^{\mu+1}.
\end{align}
\end{lemma}
\begin{proof}\textsc{.}
	Since $\norm{F(v,\nabla v)}_{Y^p}$ is defined as the maximum of $\norm{F_i(v,\nabla v)}_{Y^p}$, it suffices to show the statements of the lemma for some component $F_i$ of $F$. We restrict our attention to the proof of the Lipschitz estimate \eqref{abschaetzung}. The argument for \eqref{5} is similar and even shorter.
	By the definition of the nonlinearity $F_i$ and an application of the triangle inequality, it holds that
	\begin{align*}	\MoveEqLeft[5]
		\norm{F_i(v,\nabla v)-F_i(w,\nabla w)}_{Y^p} = \norm{ \sum \limits_{j=1}^d A_{ij}(v)\nabla v_j - \sum \limits_{j=1}^d A_{ij}(w)\nabla w_j}_{Y^p}\\
		&\leq \sum \limits_{i=1}^d  \norm{A_{ij}(v)\left(\nabla v_j -\nabla w_j\right)}_{Y^p} + \sum \limits_{i=1}^d \norm{\left(A_{ij}(v)-A_{ij}(w)\right) \nabla w_j}_{Y^p}.
	\end{align*}
	We make now use of the assumptions on the reaction matrix $A$ in \eqref{matrixbedingung} and the fact that $\|\grad w\|_{Y^p} = \|w\|_{\dot X^p}$ to estimate
			\begin{align*}	\MoveEqLeft[5]
		\norm{F_i(v,\nabla v)-F_i(w,\nabla w)}_{Y^p} \\
		&\leq  \sum \limits_{i=1}^d   \norm{A_{ij}(v)}_{L^{\infty}} \norm{v_j-w_j}_{X^p} +  \sum \limits_{i=1}^d \norm{A_{ij}(v)-A_{ij}(w)}_{L^{\infty}} \norm{v_j}_{X^p}\\
		& \lesssim d   \norm{v}_{L^{\infty}}^{\mu}   \norm{v-w}_{X^p} + d \max \left\{ \norm{v}_{L^{\infty} }^{\nu}, \norm{w}_{L^{\infty} }^{\nu} \right\} \norm{v-w}_{L^{\infty} }   \norm{v}_{X^p} \\
		& \lesssim  d \max \Big\{\norm{v}_{X^p}^{\mu},  \norm{v}_{X^p}^{\nu+1}, \norm{w}_{X^p}^{\nu+1} \Big\} \norm{v-w}_{X^p}.
	\end{align*}
	This proves \eqref{abschaetzung}. 
	\end{proof}\\
%

We now have all prerequisites to prove Theorem \ref{thm1}.

\begin{proof}\textsc{ of Theorem \ref{thm1}.} 
Let $w\in X^p$ be given, and let $T[h,w]$ be the solution to the linear problem (\ref{heateq}) with  inhomogeneity $\nabla \cdot F(w,\nabla w)$ and initial data $h$. By Proposition \ref{proposition} we obtain the estimate $\norm{T[w,h]}_{X^p} \lesssim \norm{h}_{L^{\infty}} + \norm{F(w,\nabla w)}_{Y^p}$. Applying Lemma \ref{lemmakontraktion} and using the assumptions on $h$ we furthermore have ${ \norm{h}_{L^{\infty}} + \norm{F(w, \nabla w)}_{Y^p} \lesssim \delta + d \norm{w}^{\mu+1}_{X^p}}$, and thus, combining both estimates, we get the following bound on the solution of the linear problem
\[
\norm{T[h,w]}_{X^p} \leq C \big( \delta + d \norm{w}^{\mu+1}_{X^p}\big)
\]
for some constant $C$ that we keep fixed for a moment. In order to define a contraction map, we define $\delta_0 = \left(\frac{1}{d (2C)^{\mu+1}}\right)^{\frac{1}{\mu}}$, to the effect that 
	\begin{align*}
		\norm{T[h,w]}_{X^p} \leq C \big( \delta + d (2C\delta)^{\mu+1} \big)\leq 2C\delta
	\end{align*}
for any  $w \in B_{2C\delta}(0)\subseteq X^p$, provided that $\delta \leq \delta_0$. Hence, for every such $\delta$ and every $h$ fixed,  the function $T(h, \cdot)$ maps the set $B_{2C\delta}(0)\subseteq X^p$ into itself.

Furthermore, by a similar argument, given $w_1$ and $w_2$, the linearity and Lemma \ref{lemmakontraktion} yield
	\begin{align*}
		\norm{T[w_1,h]-T[w_2,h]}_{X^p} &\leq \tilde C d \max \{\norm{w_1}_{X^p}^{\mu},\norm{w_2}_{X^p}^{\mu}, \norm{w_1}_{X^p}^{\nu+1}, \norm{w_2}_{X^p}^{\nu+1} \} \norm{w_1-w_2}_{X^p},
	\end{align*}
	for some constant $\tilde C$. Choosing  $\delta_0$ even smaller---if necessary---, we thus find the contraction estimate
	\begin{align*}
		\norm{T[w_1,h]-T[w_2,h]}_{X^p} \leq \theta \norm{w_1-w_2}_{X^p},
	\end{align*}
	for any ${{w_1, w_2\in  B_{2C\delta}(0)\subseteq X^p}}$ and some  $\theta <1$ fixed. 
	
An application of Banach's fixed point theorem thus provides a unique solution $w^* $ in $ B_{2C\delta}(0)\subseteq X^p$ to the equation $T[h,w^*]=w^*$, which is nothing but (\ref{allgemein}).
As a by-product, we also have the stability estimate \eqref{12}.
\end{proof}
\\

The idea how to prove the regularity of the solution was introduced  in \cite{1023071971426,angenent_1990} and is commonly referred to as Angenent's trick.

\begin{proof}\textsc{ of Theorem \ref{thm2}.}
	To show that the dependence of the solution  on the initial datum is of class $C^m$, $C^{\infty}$ or $C^{\omega}$, we consider the operator $L: L^{\infty}(\T^n, \R^d)\times X^p \rightarrow X^p$ defined by ${L[h,w] = w - T[h,w]}$, where $T$ is the fixed point map introduced in the proof of Theorem \ref{thm1} above.
	Defined  on $B_{\delta}(0)\times B_{2C\delta}(0) \subseteq L^{\infty}(\T^n,\R^d)\times X^p$, this map $T$ is of the same differentiability class as the nonlinearity $F(w,\nabla w)$ through $A(w)$, and so is the operator $L$ by definition. Indeed, if, for instance, $A$ is $C^1$, we notice that
	\begin{align*}
\MoveEqLeft	\left|F_i(w_1,\grad w_1) - F_i( w_2,\grad  w_2) - D_wF_i(w_2,\grad w_2) (w_1-w_2)- D_{\grad w}F_i(w_2,\grad w_2)(\grad w_1-\grad w_2)\right|\\
& \le \sum_{j} |A_{ij}(w_1) - A_{ij}(w_2) - A'_{ij}(w_2)(w_1-w_2) ||\grad w_j| \\
&\quad + \sum_j |A_{ij}'(w_2)||w_1-w_2||\grad w_1-\grad w_2|,
	\end{align*}
	and the right-hand side is a $o(\|w_1-w_2\|_{X^p})$ term, and the derivative of the fixed-point map $T[h,w]$ with respect to $w$ is given by the solution of the heat equation with with inhomogeneities $\div(D_wF_i(w,\grad w) v+ D_{\grad w}F_i(w,\grad w)\grad v)$.
	Next we observe, that $L[0,0]=0$ holds and $D_wL[0,0]= id$ is invertible.
We are thus in the position to apply the (analytic) implicit function theorem (see for example \cite{Deimling85}) to deduce the existence of balls ${B_{\hat{\delta}}(0) \subseteq L^{\infty}(\T^n,\R^d)}$ and $B_{\varepsilon}(0)\subseteq X^p$ and of a function ${S:L^{\infty}(\T^n, \R^d) \supseteq B_{\hat{\delta}}(0) \rightarrow B_{\varepsilon}(0)\subseteq X^p}$ of class $C^m$, $C^\infty$  or $C^\omega$ with $S[0]=0$ and $L[h,S[h]]=0$. For $\tilde{\delta} = \min(\delta, \hat{\delta})$ and $\tilde{\varepsilon}=\min(\varepsilon, \varepsilon_0)$ we obtain, due to the definition of $L$, a unique solution $w^* \in B_{\tilde{\varepsilon}}(0)\subseteq X^p$ that depends of class $C^m$, class $C^\infty$ or analytically on the initial data $h \in B_{\tilde{\delta}}(0) \subseteq L^{\infty}(\T^n,\R^d)$.

	Finally, we show the regularity of the solution $w^*$. For this purpose, we define a translation operator $\Psi_{s,a}:\R\times\R^n\rightarrow \R\times\R^n$ by
	\begin{align*}
	\Psi_{s,a}(t,x) \coloneqq (st,x+t^{\frac{1}2}a)\ \ \ \text{ and set } \ \ \ w^*_{s,a} \coloneqq w^* \circ \Psi_{s,a}.
	\end{align*}
	We notice that $w^*_{s,a}$ solves the equation
	\begin{align*}
		\partial_t w^*_{s,a} -\Delta w^*_{s,a} =\nabla \cdot F_{s,a}(w^*_{s,a}, \nabla w^*_{s,a}),
	\end{align*}
	where
	\begin{align*}
		F_{s,a}(w,\nabla w) \coloneqq s F(w,\nabla w)+(s-1)\nabla w + \frac{1}{2} a{ t^{-\frac{1}{2}}} w.
	\end{align*}
	By definition, it holds that $F_{1,0}(w, \nabla w)= F(w, \nabla w)$. Let $T_{s,a}[h,w]$ denote the solution to the linear problem with inhomogeneity $\nabla \cdot F_{s,a}(w,\grad w)$ and initial data $h$. Since $\|a{ t^{-\frac{1}{2}}}w\|_{Y^p} \lesssim \norm{w}_{L^{\infty}}$, Lemma \ref{lemmakontraktion} holds true for $F_{s,a}$ as well. We set, similarly as above, $L_{s,a}[h,w] = w - T_{s,a}[h,w]$. Again it holds, that $L_{1,0}[0,0]=0$ and $D_wL_{1,0}[0,0]=id$.
	 Another application of the implicit function theorem thus yields the existence of two numbers $\lambda>0$ and $\delta_0>0$ as well as a function  $S_{s,a}[h]=S[s,a,h]$ from  $B_{\lambda}(1) \times B_{\lambda}(0)\times B_{\delta}(0) \subseteq \R \times \R^n \times L^{\infty}(\T^n,\R^d)$ to $B_{2C\delta}(0)\subseteq X^p$ of class $C^m$, $C^\infty$ or $C^\omega$ for every $\delta \leq \delta_0$. The function $S_{s,a}$ satisfies $L_{s,a}[h,S_{s,a}[h]]=0$ and thus $S_{s,a}[h]=T_{s,a}[h,S_{s,a}[h]]$.
	From the above uniqueness results we deduce that $S_{s,a}[h]= S[h] \circ \Psi_{s,a}$. Moreover, since ${S[h](0,\cdot) = h = S_{s,a}[h](0,\cdot)}$ and $L_{s,a}[h,S[h]\circ \Psi_{s,a}]=0$, it holds that the dependence of $S[h]\circ \Psi_{s,a}(t,x)$  on the parameters $a$ and $s$ is of class $C^m$, $C^\infty$ or $C^\omega$ in a small neighbourhood of $(1,0) \in \R\times \R^n$.
	For finite $t$ we can calculate the derivatives,
	\begin{align*}
	\partial_s^k \partial_a^{\beta}\big|_{(s,a)=(1,0)}S[h]\circ \Psi_{s,a}(t,x) = t^{k+\frac{|\beta|}2}\partial_t^k\partial_x^{\beta}w(t,x).
	\end{align*}
	This shows, that $S[h]$ and thereby $w^*$ as well is of class $C^m$, class $C^\infty$ or analytic in space and time for every $x \in \T^n$ and every $0<t<\infty$. 	Since $\norm{S[h]\circ \Psi_{s,a}}_{L^{\infty} } \leq \norm{S[h]}_{X^p} \lesssim \norm{h}_{L^{\infty} }$, we deduce \eqref{7}.
	
	To cover the analytic case, it only remains to recall the elementary fact that we can estimate arbitrary derivatives of an analytic function $f$ locally by $		|f^{(j)}(y)| \leq C \frac{j!}{\Gamma^j}\norm{f}_{L^{\infty}} $ for some positive reals $C$ and $  \Gamma	$. This concludes the proof of Theorem \ref{thm2}.
	\end{proof}\\

We finally turn to the proof of  Theorem \ref{thm3}. Thanks to the results obtained so far for the general systems, it is enough to show that solutions to \eqref{bewiesenessystem} satisfy the partition of unity condition \eqref{bedingungsummeW}. For this purpose, it is convenient to  truncate the nonlinearities. Inspired by   \cite{articlepietschmann}, we consider
\begin{align}\label{modifiziertesproblem}
\begin{cases}
\partial_t w_i - \Delta w_i = \nabla \cdot \hat{F}_i(w,\grad w)  &  \text{ in } \left(0,\infty\right) \times \R^n,\\
w_i(0,\cdot)=h_i & \text{ in } \R^n,
\end{cases} \qquad & i = 1,\dots d,
\end{align}
with nonlinearities
\[
\hat{F}_i (w,\grad w) \coloneqq     \sum \limits_{i=1, j \neq i}^d \alpha_{ij} (\hat{w}_j\nabla w_i - \hat{w}_i \nabla w_j) ,
\]
where $\hat{w}_i$ is obtained from   $w_i$ by restriction to the range $[0,\delta]$, i.e., $\hat{w}_i \coloneqq \max \left(0, \min(\delta,w_i)\right)$. We have to show that solutions to the truncated problem satisfy \eqref{bedingungsummeW} and that $\hat w_i = w_i$ to deduce statement of Theorem \ref{thm3}.

\begin{proof}\textsc{ of Theorem \ref{thm3}.} The general well-posedness result of Theorem \ref{thm1} applies to the modified problem \eqref{modifiziertesproblem}, and we see that  $\delta_0 $ has to be chosen much smaller than $ 1/d$ by a closer inspection of the proof. We denote the unique solution to \eqref{modifiziertesproblem} by $w^*$.

Our goal is to show, that $w^*$ fulfils the partition of unity condition (\ref{bedingungsummeW}). Therefore we start by adding up all $d$ equations of (\ref{modifiziertesproblem}). Due to the symmetry condition $K_{ij}=K_{ji}$ imposed in \eqref{2}, which is inherited by the $\alpha_{ij}$'s, this leads to considering the homogeneous heat equation
\begin{align*}
\begin{cases}
\partial_t W - \Delta W = 0 &\text{ in } (0,\infty) \times \T^n,\\
W(0, \cdot)= \delta & \text{ in } \T^n,
\end{cases}
\end{align*}
for $W\coloneqq \sum \limits_{i=1,\dots,d} w_i^*$, which is solved by $W=\delta$.

To show, that the $w_i^*$'s stay nonnegative, we consider the negative parts of $w_i^*$, namely $w_i^{*-} \coloneqq \min (0,w_i^*)$. Multiplying the $i$th equation of (\ref{modifiziertesproblem}) by $w_i^{*-}$ and integrating over $\T^n$ leads to
\begin{align*}
\MoveEqLeft[6]
\inttn w_i^{*-} \partial_t w_i^* dx - \inttn w_i^{*-} \Delta w_i^*dx \\
&= \inttn w_i^{*-} \nabla \cdot \big( \sum \limits_{j=1, j \neq i}^d \alpha_{ij} \hat{w}_j^*\nabla w_i^* \big)dx - \inttn w_i^{*-} \nabla \cdot \big( \sum \limits_{j=1, j \neq i}^d \alpha_{ij} \hat{w}^*_i\nabla w_j^* \big)dx.
\end{align*}
By a multiple integration by parts, taking into account   that $\hat{w}^*_i \nabla w_i^{*-} = 0$ and $w_i^{*-} \hat{w}^*_i =0$, we derive the energy identity
\begin{align}\label{positiv}
0 =  \frac{1}{2} \frac{d}{dt} \norm{w_i^{*-}}_{L^2(\T^n)}^2+ \inttn |\nabla w_i^{*-}|^2 \big( 1 + \sum \limits_{j=1, j \neq i}^d \alpha_{ij}\hat{w}^*_j \big)dx.
\end{align}
Since $w^* \in B_{2C\delta}(0) \subseteq X^p$, we know $|w_j^*| \leq 2C\delta$ for every $j$ and therefore $\hat{w}^*_j \in [0, 2C\delta]$, where $2C\delta < \frac{1}{2(d-1)C}$. We can assume, that the positive constant $C$ is greater than one and thus, using $\alpha_{ij} \in [-1,1]$, we obtain
\begin{align*}
1 + \sum \limits_{j=1, j \neq i}^d \alpha_{ij}\hat{w}^*_j \geq 1 -\sum \limits_{j=1, j \neq i}^d \hat{w}^*_j > 1 - (d-1)\frac{1}{2(d-1)}=\frac{1}{2}.
\end{align*}
Hence the second term in (\ref{positiv}) is nonnegative. This provides that the $L^2$-norm of $w_i^{*-}$ decreases in time. Together with the fact $h_i = w_i^*(0,\cdot)$ is nonnegative for every $i=1,\dots,d$, we obtain $w_i^-=0$ and thus  $w_i^*\geq 0$ almost everywhere in $(0,\infty) \times \T^n$. 

We have thus seen that $w^*$ solves partition of unity condition \eqref{bedingungsummeW} almost everywhere in $(0,\infty) \times \T^n$, and thus $w^* = \hat w^*$ almost everywhere. It remains to note that thanks  to the regularity established in Theorem \ref{thm2} and the continuity of the nonlinearity $\hat{F}(w,\nabla w)$, the solution $w^*$ is continuous as well and this property expands to the whole domain $(0,\infty)\times \T^n$. 
\end{proof} 
\section*{Acknowledgement} 
This work is funded by the Deutsche Forschungsgemeinschaft (DFG, German Research Foundation) under Germany's Excellence Strategy EXC 2044 --390685587, Mathematics M\"unster: Dynamics--Geometry--Structure.

\bibliography{mybib}{}
\bibliographystyle{abbrv}

\newpage

\end{document}